\documentclass[a4paper,leqno]{amsart}

\usepackage{latexsym}
\usepackage[english]{babel}
\usepackage{fancyhdr}
\usepackage[mathscr]{eucal}
\usepackage{amsmath}
\usepackage{mathrsfs}
\usepackage{amsthm}
\usepackage{amsfonts}
\usepackage{amssymb}
\usepackage{amscd}
\usepackage{bbm}
\usepackage{graphicx}
\usepackage{graphics}
\usepackage{latexsym}
\usepackage{color}
\usepackage{calc}

\newcommand{\ud}{\mathrm{d}}

\newcommand{\ii}{\mathrm{i}}
\newcommand{\cH}{\mathcal{H}}

\theoremstyle{plain}
\newtheorem{theorem}{Theorem}[section]
\newtheorem{lemma}[theorem]{Lemma}
\newtheorem{corollary}[theorem]{Corollary}
\newtheorem{proposition}[theorem]{Proposition}

\theoremstyle{definition}

\newtheorem{remark}[theorem]{Remark}
\newtheorem*{remark*}{Remark}

\numberwithin{equation}{section}

\begin{document}

\title[Self-adjoint extensions with Friedrichs lower bound]
{Self-adjoint extensions \\ with Friedrichs lower bound}
\author[M.~Gallone]{Matteo Gallone}
\address[M.~Gallone]{International School for Advanced Studies -- SISSA \\ via Bonomea 265 \\ 34136 Trieste (Italy).}
\email{mgallone@sissa.it}
\author[A.~Michelangeli]{Alessandro Michelangeli}
\address[A.~Michelangeli]{Institute for Applied Mathematics, University of Bonn \\ Endenicher Allee 60 \\ 
D-53115 Bonn (Germany).}
\email{michelangeli@iam.uni-bonn.de}


\begin{abstract}
 We produce a simple criterion and a constructive recipe to identify those self-adjoint extensions of a lower semi-bounded symmetric operator on Hilbert space which have the same lower bound as the Friedrichs extension. Applications of this abstract result to a few instructive examples are then discussed.
\end{abstract}

\date{\today}

\subjclass[2000]{}

\thanks{Partially supported by 
the Istituto Nazionale di Alta Matematica (INdAM) and 
the Alexander von Humboldt foundation.}

\maketitle

\vspace{-0.3cm} 


\section{Motivation}\label{sec:motivation}

We start with a familiar example. In the Hilbert space $\cH=L^2(0,1)$ let us consider the densely defined, closed, and symmetric operator
\begin{equation}\label{eq:Slaplace01}
  S\;=\;-\frac{\ud^2}{\ud x^2}\,,\quad\mathcal{D}(S)\;=\;H^2_0(0,1)\;=\;
 \left\{f\in H^2(0,1)\left|\!
 \begin{array}{c}
  f(0)=0=f(1) \\
  f'(0)=0=f'(1)
 \end{array}\!\!
 \right.\right\}.
\end{equation}
$S$ is actually the operator closure of the negative Laplacian defined on $C^\infty_0(0,1)$.
Here and in the following $\mathcal{D}(R)$ denotes the domain of the operator $R$ acting on $\cH$, and if $R$ is symmetric we denote by
\begin{equation}
 \mathfrak{m}(R)\;:=\;\inf_{\substack{ f\in\mathcal{D}(R) \\  f\neq 0 }}\frac{\langle f,Rf\rangle}{\|f\|_2^2}\;\in[-\infty,+\infty)
\end{equation}
the largest lower bound of $R$. When $\mathfrak{m}(R)>-\infty$ one says that $R$ is semi-bounded from below. Poincar\'e inequality implies that $S$ is semi-bounded from below with
\begin{equation}
 \mathfrak{m}(S)\;=\;\pi^2\,.
\end{equation}

Now, $S$ is symmetric but not self-adjoint, for
\begin{equation}
 S^*\;=\;-\frac{\ud^2}{\ud x^2}\,,\qquad \mathcal{D}(S^*)\;=\;H^2(0,1)\,.
\end{equation}
Thus, $S$ admits a multiplicity (in fact, a four-real-parameter family) of distinct self-adjoint extensions, which are all restrictions of $S^*$. Among them, the Friedrichs extension $S_F$ is the one with domain
\begin{equation}\label{eq:DSFlaplacian}
 \mathcal{D}(S_F)\;=\;H^2(0,1)\cap H^1_0(0,1)\;=\;\big\{ f\in H^2(0,1)\,\big|\, f(0)=0=f(1) \big\}\,,
\end{equation}
namely the Dirichlet (negative) Laplacian. Let us recall that abstractly speaking the Friedrichs extension of a lower semi-bounded symmetric operator $S$ is the \emph{only} self-adjoint extension of $S$ with the property
\begin{equation}
 \mathcal{D}(S_F)\;\subset\;\mathcal{D}[S]\,,
\end{equation}
that is, with operator domain contained in the form domain of $S$. Here and in the following $\mathcal{D}[R]$ denotes the form domain of a lower semi-bounded symmetric operator $R$, or also of a self-adjoint operator $R$ (see, e.g., \cite[Chapt.~10]{schmu_unbdd_sa}); in the present case
\begin{equation}
 \mathcal{D}[S]\;=\;\overline{\mathcal{D}(S)}^{\|\,\|_{H^1}}=\;H^1_0(0,1)\,,
\end{equation}
and obviously $\mathcal{D}[S]=\mathcal{D}[S_F]$. 
It is also a general property of the Friedrichs extension the fact that $S_F\geqslant \widetilde{S}$ for any other $\widetilde{S}=\widetilde{S}^*\supset S$, namely $S_F$ is the largest of all the self-adjoint extensions of $S$ in the sense of operator form ordering.

As well known, as follows from \eqref{eq:DSFlaplacian}, 
$S_F$ is diagonalizable over an orthonormal basis $\{\sqrt{2}\sin n\pi x\,|\,n\in\mathbb{N}\}$ of eigenfunctions, with (simple and pure point) spectrum
\begin{equation}
 \sigma(S_F)\;=\;\{n^2\pi^2\,|\,n\in\mathbb{N}\}
\end{equation}
Thus,
\begin{equation}
 \mathfrak{m}(S_F)\;=\;\pi^2\;=\;\mathfrak{m}(S)\,,
\end{equation}
which actually expresses a completely general fact: the Friedrichs extension of a lower semi-bounded operator preserves the lower bound. Whereas self-adjoint extensions of $S$ cannot increase the lower bound, in general they decrease it. For instance, in the example under consideration, the extension $S_P$ with periodic boundary conditions, namely with domain
\begin{equation}
 \mathcal{D}(S_P)\;=\;\big\{ f\in H^2(0,1)\,\big|\, f(1)=f(0)\,,f'(1)=f'(0) \big\}\,,
\end{equation}
has spectrum
\begin{equation}
 \sigma(S_P)\;=\;\{n^2\pi^2\,|\,n\in\mathbb{N}_0\}\,,\qquad\textrm{whence}\qquad \mathfrak{m}(S_P)\;=\;0\,.
\end{equation}
Yet, the extension $S_A$ with anti-periodic boundary conditions, namely with
\begin{equation}
 \mathcal{D}(S_A)\;=\;\big\{ f\in H^2(0,1)\,\big|\, f(1)=-f(0)\,,f'(1)=-f'(0) \big\}\,,
\end{equation}
does preserve the lower bound of $S$. Indeed,
\begin{equation}
 \sigma(S_A)\;=\;\{(2n+1)^2\pi^2\,|\,n\in\mathbb{N}_0\}\,,\qquad\textrm{whence}\qquad \mathfrak{m}(S_P)\;=\;\pi^2\,.
\end{equation}

This occurrence is well known: a lower semi-bounded symmetric operator may admit self-adjoint extensions other than the Friedrichs, with the same bottom of the Friedrichs spectrum. Actually this is not typical of symmetric operators with deficiency index $2$ only, as was the case for $S$ here. In Section \ref{sec:applications} also examples with deficiency indices $1$ will be recalled and discussed. By standard direct sum, these examples also cover the case of infinite deficiency indices.

Now, while the possibility of non-Friedrichs self-adjoint extensions with the same Friedrichs lower bound is folk knowledge, we are not aware of an explicit operator-theoretic explanation of this phenomenon, nor of a  characterisation in terms of transparent conditions which, once they are met, allow to construct all extensions with such a feature.

In this note we present a simple criterion and a constructive recipe to identify those self-adjoint extensions of a lower semi-bounded symmetric operator on Hilbert space which have the same lower bound as the Friedrichs extension. The abstract main results, Theorems \ref{thm:main1}, \ref{thm:main2}, and \ref{thm:main3} below, are discussed in Section \ref{sec:abstractresults}, and illustrative concrete examples where such results can be applied to are then presented in Section \ref{sec:applications}.

\section{Abstract results}\label{sec:abstractresults}

Let $\cH$ be a Hilbert space (over $\mathbb{R}$ or $\mathbb{C}$, with scalar product $\langle\cdot,\cdot\rangle$ anti-linear in the first entry, and with norm $\|\,\|$) and let $S$ be a densely defined, symmetric, semi-bounded operator on $\cH$ with lower bound $\mathfrak{m}(S)$. $S$ in not necessarily closed. For clarity of the presentation we shall assume non-restrictively $\mathfrak{m}(S)>0$. This implies that $S_F^{-1}$ is everywhere defined and bounded on $\cH$. It will be clear both from this abstract discussion and from the applications in Section \ref{sec:applications} that the case of general finite $\mathfrak{m}(S)$ can be covered by suitably shifting $S$ to $S-\mu\mathbbm{1}$ with $\mu<\mathfrak{m}(S)$.

Unless such $S$ is already essentially self-adjoint, it admits non-trivial self-adjoint extensions. In this case $\ker S^*$, the deficiency space for $S$, is non-trivial either. Standard extension schemes produce convenient classifications of the whole family of extensions. It can be shown within the modern theory of boundary triplets \cite{Behrndt-Hassi-deSnoo-boundaryBook}, or equivalently the classical `universal' parametrization by Grubb \cite{Grubb-1968}, and in fact the very original extension theory by Kre{\u\i}n \cite{Krein-1947}, Vi\v{s}ik \cite{Vishik-1952}, and Birman \cite{Birman-1956}, that the extensions of $S$ can be labelled as follows.

\begin{theorem}\label{thm:KVB-parametrisation}
 There is a one-to-one correspondence between the  family of all self-adjoint extensions of  $S$ on $\cH$ and the family of the self-adjoint operators on Hilbert subspaces of $\ker S^*$.
\begin{itemize}
 \item[(i)] If $T$ is any such operator, in the correspondence $T\leftrightarrow S_T$ each self-adjoint extension $S_T$ of $S$ is given by
\begin{equation}\label{eq:ST}
\begin{split}
S_T\;&=\;S^*\upharpoonright\mathcal{D}(S_T) \\
\mathcal{D}(S_T)\;&=\;\left\{f+S_F^{-1}(Tv+w)+v\left|\!\!
\begin{array}{c}
f\in\mathcal{D}(\overline{S})\,,\;v\in\mathcal{D}(T) \\
w\in\ker S^*\cap\mathcal{D}(T)^\perp
\end{array}\!\!
\right.\right\}.
\end{split}
\end{equation}
\item[(ii)] If $S_T$ is a semi-bounded (not necessarily positive) self-adjoint extension of $S$, then
\begin{equation}\label{eq:D[SB]_Tversion}
\mathcal{D}[T]\;=\; \mathcal{D}[S_T]\,\cap\,\ker S^*
\end{equation}
and
 \begin{equation}\label{eq:decomposition_of_form_domains_Tversion}
 \begin{split}
 \mathcal{D}[S_T]\;&=\;\mathcal{D}[S_F]\,\dotplus\,\mathcal{D}[T] \\
 S_T[f+v,f'+v']\;&=\;S_F[f,f']\,+\,T[v,v'] \\
 &\forall f,f'\in\mathcal{D}[S_F],\;\forall v,v'\in\mathcal{D}[T]\,.
 \end{split}
\end{equation}
As a consequence,
\begin{equation}\label{eq:extension_ordering_Tversion}
S_{T_1}\,\geqslant\,S_{T_2}\qquad\Leftrightarrow\qquad T_1\,\geqslant\,T_2\,.
\end{equation}
\item[(iii)] If $\mathfrak{m}(T)>-\mathfrak{m}(S)$, then
 \begin{equation}\label{eq:bounds_mS_mB_Tversion}
 \mathfrak{m}(T)\;\geqslant\; \mathfrak{m}(S_T)\;\geqslant\;\frac{\mathfrak{m}(S) \,\mathfrak{m}(T)}{\mathfrak{m}(S)+\mathfrak{m}(T)}\,.
 \end{equation}
\end{itemize}
\end{theorem}

Theorem \ref{thm:KVB-parametrisation} collects results that are proved, e.g., in \cite[Chapt.~13]{Grubb-DistributionsAndOperators-2009}, \cite[Chapt.~14]{schmu_unbdd_sa}, and \cite[Sect.~3]{GMO-KVB2017}.

For convenience, let us denote by $\mathcal{S}(\mathcal{K})$ the collection of all self-adjoint operators defined in Hilbert subspaces of a given Hilbert space $\mathcal{K}$: Theorem \ref{thm:KVB-parametrisation} states that the self-adjoint extensions of $S$ are all of the form $S_T$ for some $T\in\mathcal{S}(\ker S^*)$.

The Friedrichs extension of $S$ can be expressed in terms of the classical decomposition formula (see, e.g., \cite[Sect.~2.2]{GMO-KVB2017})
\begin{equation}
 \mathcal{D}(S_F)\;=\;\mathcal{D}(\overline{S})\dotplus S_F^{-1}\ker S^*\,.
\end{equation}
Therefore, $S_F$ is recovered from the general parametrisation \eqref{eq:ST} or \eqref{eq:decomposition_of_form_domains_Tversion} with the choice $\mathcal{D}[T]=\{0\}$ (thus, formally, ``$T=\infty$'').

An ancillary result that tends to be somehow less highlighted, but which is most relevant for our discussion, is the following.

\begin{theorem}\label{thm:boundsSTboundT}
 If, with respect to the notation of \eqref{eq:ST}, $S_T$ is a self-adjoint extension of $S$, and if $\mu<\mathfrak{m}(S)$, then
\begin{equation}\label{eq:SBsmbb-iff-invBsmbb_Tversion}
\begin{split}
\langle g,S_T g\rangle\;&\geqslant\;\mu\,\|g\|^2\qquad\forall g\in\mathcal{D}(S_T) \\
& \Updownarrow \\
\langle v,T v\rangle\;\geqslant\;\mu\|v\|^2+\:&\mu^2\langle v,(S_F-\mu\mathbbm{1})^{-1} v\rangle\qquad\forall v\in\mathcal{D}(T)\,.
\end{split}
\end{equation}
\end{theorem}

As an immediate consequence, Theorem \ref{thm:boundsSTboundT} reproduces the inequality $\mathfrak{m}(T)\geqslant \mathfrak{m}(S_T)$ for any semi-bounded $S_T$ and shows, in particular, that positivity or strict positivity of the bottom of $S_T$ is equivalent to the same property for $T$, that is,
 \begin{equation}\label{eq:positiveSBiffpositveB-1_Tversion}
 \begin{split}
 \mathfrak{m}(S_T)\;\geqslant \;0\quad&\Leftrightarrow\quad \mathfrak{m}(T)\;\geqslant\; 0 \\
 \mathfrak{m}(S_T)\;> \;0\quad&\Leftrightarrow\quad \mathfrak{m}(T)\;>\; 0\,.
 \end{split}
 \end{equation}


To make this presentation self-contained, and for later convenience, let us deduce Theorem \ref{thm:boundsSTboundT} from Theorem \ref{thm:KVB-parametrisation}. To this aim, let us first single out a simple operator-theoretic property.

\begin{lemma}\label{lem:A_A-1}
If $A$ is a self-adjoint operator on a Hilbert space $\cH$ with positive bottom ($\mathfrak{m}(A)>0$), then
\[
\sup_{f\in\mathcal{D}(A)}\frac{\;|\langle f,h\rangle|^2}{\langle f,Af\rangle}\;=\;\langle h,A^{-1}h\rangle\qquad\forall h\in\cH.
\]
\end{lemma}

\begin{proof}
Setting $g:=A^{1/2}f$ one has
\[
\sup_{f\in\mathcal{D}(A)}\frac{\;|\langle f,h\rangle|^2}{\langle f,Af\rangle}\;=\;\sup_{g\in\cH}\frac{\;|\langle A^{-1/2}g,h\rangle|^2}{\|g\|^2}\;=\;\sup_{\|g\|=1}|\langle g,A^{-1/2}h\rangle|^2
\]
and since $|\langle g,A^{-1/2}h\rangle|$ attains its maximum for $g=A^{-1/2}h/\|A^{-1/2}h\|$, the conclusion then follows.
\end{proof}

\begin{proof}[Proof of Theorem \ref{thm:boundsSTboundT}]
 For generic $f\in\mathcal{D}(S_F)$ and $v\in\mathcal{D}(T)$, one has $g:=f+v\in\mathcal{D}(S_T)$ and
 \[
  S_T[g]\;=\;\langle f,S_F f\rangle + \langle v, T v\rangle\,.
 \]
 Thus, $S_T\geqslant \mu\mathbbm{1}$ is tantamount as requiring for all such $g$'s that
\[
\langle f,S_F f\rangle+\langle v,T v\rangle\;\geqslant\;\mu\,\big( \langle f,f\rangle+\langle f,v\rangle+\langle v,f\rangle+\langle v,v\rangle\big)
\]
whence also, replacing $f\mapsto\lambda f$, $v\mapsto\gamma v$,
\[
\begin{split}
\big(\langle f,S_F f\rangle&-\mu\|f\|^2\big)\,|\lambda|^2-\mu\langle f,v\rangle\overline{\lambda}\gamma-\mu\langle v,f\rangle \lambda\overline{\gamma} \\
&\qquad +\big(\langle v,T v\rangle-\mu\|v\|^2\big)\,|\gamma|^2\;\geqslant\;0\qquad\forall \lambda,\gamma\in\mathbb{C}\,.
\end{split}
\]
Since $\mu<\mathfrak{m}(S)$, and hence $\langle f,S_F f\rangle-\mu\|f\|^2> 0$, last inequality holds true if and only if
\[\tag{*}\label{eq:iffineq}
\mu^2|\langle f,v\rangle|^2\;\leqslant\;\big(\langle v,T v\rangle-\mu\|v\|^2\big)\,\big(\langle f,S_F f\rangle-\mu\|f\|^2\big)
\]
for arbitrary $f\in\mathcal{D}(S_F)$ and $v\in\mathcal{D}(T)$. By re-writing \eqref{eq:iffineq} as
\[
\langle v,T v\rangle-\mu\|v\|^2 \;\geqslant\;\mu^2\,\frac{|\langle f,v\rangle|^2}{\langle f,(S_F-\mu\mathbbm{1})f\rangle}
\]
and by the fact that the above inequality is valid for arbitrary $f\in\mathcal{D}(S_F)$ and hence holds true also when the supremum over such $f$'s is taken, Lemma \ref{lem:A_A-1} then yields
\[
\langle v,T v\rangle-\mu\|v\|^2 \;\geqslant\;\mu^2\langle v,(S_F-\mu\mathbbm{1})^{-1} v\rangle\,,
\]
 which completes the proof. 
\end{proof}

With these abstract results at hand, let us now turn to the identification of the non-Friedrichs extensions with the same Friedrichs lower bound.

It is worth observing that inequality \eqref{eq:bounds_mS_mB_Tversion} is not informative in this respect: indeed, owing to \eqref{eq:bounds_mS_mB_Tversion}, a \emph{sufficient} condition for the bottom of $S_T$ to equal the bottom of $S_F$ would be to impose $\mathfrak{m}(S)\mathfrak{m}(T)/(\mathfrak{m}(S)+\mathfrak{m}(T))\geqslant \mathfrak{m}(S)$, but such inequality is only satisfied, in the form of an identity, when $\mathfrak{m}(T)=\infty$, therefore the above sufficient condition only selects $S_T=S_F$, the Friedrichs extension.

We rather focus on \eqref{eq:SBsmbb-iff-invBsmbb_Tversion} from Theorem \ref{thm:boundsSTboundT}. There, the operator $S_F-\mu\mathbbm{1}$ is invertible with everywhere bounded inverse on the whole $\cH$: indeed, $\mu<\mathfrak{m}(S)$ and then $\mathfrak{m}(S_F-\mu\mathbbm{1})>0$. Instead, $S_F-\mathfrak{m}(S)\mathbbm{1}$ fails to be invertible on $\cH$, because its bottom is by construction equal to zero.

The informal idea now is that even if $(S_F-\mathfrak{m}(S)\mathbbm{1})^{-1}$ cannot be defined as a bounded operator \emph{on the whole} $\cH$, yet it makes sense on $\mathrm{ran}(S_F-\mathfrak{m}(S)\mathbbm{1})$, and if it happens that the latter space has a non-trivial intersection with $\ker S^*$, then there are non-zero vectors $v\in \mathrm{ran}(S_F-\mathfrak{m}(S)\mathbbm{1})\cap\ker S^*$ on which $\langle v,(S_F-\mathfrak{m}(S)\mathbbm{1})^{-1}v\rangle$ is unambiguously defined and hence the right-hand side of the second expression in \eqref{eq:SBsmbb-iff-invBsmbb_Tversion} is meaningful also when $\mu=\mathfrak{m}(S)$. Moreover, if on such $v$'s one can define an operator $T\in\mathcal{S}(\ker S^*)$ satisfying \eqref{eq:SBsmbb-iff-invBsmbb_Tversion}  when $\mu=\mathfrak{m}(S)$, then by suitably exploiting the limit $\mu\uparrow \mathfrak{m}(S)$ this should give a characterisation of the condition $\mathfrak{m}(S_T)\geqslant \mathfrak{m}(S)$, which is equivalent to $\mathfrak{m}(S_T)= \mathfrak{m}(S)$, as $S_F\geqslant S_T$, the Friedrichs extension is the largest of all self-adjoint extensions of $S$.

By elaborating on such idea we finally come to our main results, Theorems \ref{thm:main1}, \ref{thm:main2}, and \ref{thm:main3} below.

Clearly, underlying \eqref{eq:SBsmbb-iff-invBsmbb_Tversion} is the quadratic form language, so the actual operator to possibly invert in some subspace of $\ker S^*$ is rather $(S_F-\mathfrak{m}(S)\mathbbm{1})^{1/2}$, a positive self-adjoint operator with zero lower bound.

In this respect, as $S_F$ is self-adjoint on $\cH$, and so is $S_F-\mathfrak{m}(S)\mathbbm{1}$ with lower bound zero, then upon decomposing 
\[
 \cH\;=\;\overline{\mathrm{ran}(S_F-\mathfrak{m}(S)\mathbbm{1})}\oplus\ker(S_F-\mathfrak{m}(S)\mathbbm{1})
\]
the negative powers $(S_F-\mathfrak{m}(S)\mathbbm{1})^{-\delta}$, $\delta>0$, are naturally defined as self-adjoint operators on the Hilbert subspace $\overline{\mathrm{ran}(S_F-\mathfrak{m}(S)\mathbbm{1})}$, or also on the whole $\cH$ upon extension by zero on $\ker(S_F-\mathfrak{m}(S)\mathbbm{1})$.

In the first statement we \emph{characterise} the occurrence of non-Friedrichs extensions with the same Friedrichs lower bound.

\begin{theorem}\label{thm:main1}
 Let $S$ be a densely defined and symmetric operator on a given Hilbert space $\cH$ with lower bound $\mathfrak{m}(S)>0$. Necessary and sufficient condition for $S$ to admit self-adjoint extensions other then the Friedrichs extensions and with the same lower bound $\mathfrak{m}(S)$ is that 
 \begin{equation}\label{eq:maincond}
  \mathrm{ran}(S_F-\mathfrak{m}(S)\mathbbm{1})^{1/2}\cap\ker S^*\;\neq\;\{0\}\,.
 \end{equation}
\end{theorem}

In the applications both $\ker S^*$ and $\mathrm{ran}(S_F-\mathfrak{m}(S)\mathbbm{1})^{1/2}$ are in general spaces that one can qualify rather explicitly. Thus, condition \eqref{eq:maincond} is practically manageable and qualify the operator-theoretic \emph{mechanism} for non-Friedrichs extensions with the Friedrichs lower bound. In Section \ref{sec:applications} we shall give examples of that.

Our next result concerns the actual \emph{recipe} to construct such extensions, when \eqref{eq:maincond} is matched, thus in practice how to identify the corresponding extension parameters $T$ in $\mathcal{S}(\ker S^*)$. We shall use the customary notation of square brackets for the domain $\mathcal{D}[q]$ of a quadratic form $q$ on $\cH$ and for the evaluation $q[v]$ on elements of its domain; as usual, we shall denote by $q[v_1,v_2]$ the evaluation of the corresponding sesquilinear form defined by polarisation.

\begin{theorem}\label{thm:main2}
 Same assumptions as in Theorem \ref{thm:main1}, and assume further that condition \eqref{eq:maincond} is satisfied.
\begin{itemize}
 \item[(i)] The expression 
 \begin{equation}\label{eq:defq}
  \begin{split}
   \mathcal{D}[q]\;&:=\;\mathrm{ran}(S_F-\mathfrak{m}(S)\mathbbm{1})^{\frac{1}{2}}\cap\ker S^* \\
   q[v]\;&:=\;\mathfrak{m}(S)\|v\|^2+\mathfrak{m}(S)^2\big\| (S_F-\mathfrak{m}(S)\mathbbm{1})^{-\frac{1}{2}}v\big\|^2\,.
  \end{split}
 \end{equation}
 defines a symmetric, closed, and strictly positive quadratic form $q$. 
 \item[(ii)] Let $T_q$ be the operator on the Hilbert subspace $\overline{\mathcal{D}[q]}$ uniquely associated with $q$. Then $T_q\in\mathcal{S}(\ker S^*)$.
 \item[(iii)] For any $T\in\mathcal{S}(\ker S^*)$ with $T\geqslant T_q$, the corresponding self-adjoint extension $S_T$ of $S$ (Theorem \ref{thm:KVB-parametrisation}) has the property
 \begin{equation}\label{eq:samebottom}
  \mathfrak{m}(S_T)\;=\;\mathfrak{m}(S)\,.
 \end{equation}
 \item[(iv)] Any self-adjoint extension $S_T$ of $S$ satisfying \eqref{eq:samebottom} corresponds to an extension parameter $T\in\mathcal{S}(\ker S^*)$ with $T\geqslant T_q$.
\end{itemize}
\end{theorem}

By definition, in \eqref{eq:defq} the vector $u_\circ=(S_F-\mathfrak{m}(S)\mathbbm{1})^{-\frac{1}{2}}v$ is the minimal norm solution $u=u_\circ$ to $(S_F-\mathfrak{m}(S)\mathbbm{1})^{\frac{1}{2}}u=v$.

In view of the general classification of Theorem \ref{thm:KVB-parametrisation}, the above results admit a natural corollary that it is worth stating as a separate theorem. It is convenient to introduce the meaningful terminology `\emph{top extensions}' for all those $S_T$'s with $\mathfrak{m}(S_T)=\mathfrak{m}(S)$ (in particular, $S_F$ is a top extension), and `\emph{least-top extension}' for the extension $S_{LT}:=S_{T_q}$.

\begin{theorem}\label{thm:main3}
 Same assumptions as in Theorem \ref{thm:main1}. Each top extension $S^{\mathrm{top}}$ of $S$ satisfies
 \begin{equation}
  S_F\;\geqslant\;S^{\mathrm{top}}\;\geqslant\;S_{LT}
 \end{equation}
 in the sense of operator form ordering. Each such extension is of the form $S^{\mathrm{top}}=S_T$ for some $T\in\mathcal{S}(\ker S^*)$ with $T\geqslant T_q$, where $T_q$ is qualified in Theorem \ref{thm:main2}(ii), and they are all ordered with $T$ in the sense of \eqref{eq:extension_ordering_Tversion}.

\end{theorem}

\begin{proof}[Proof of Theorem \ref{thm:main1}, necessity part]~
 
 Let $S_T$ be a self-adjoint extension of $S$, labelled by some $T\in\mathcal{S}(\ker S^*)$, with the property $\mathfrak{m}(S_T)=\mathfrak{m}(S_F)$ and $S_T\neq S_F$. Let $(\mu_n)_{n\in\mathbb{N}}$ be an increasing sequence of real numbers such that $\mu_n<\mathfrak{m}(S)$ $\forall n$ and $\mu_n\to\mathfrak{m}(S)$ as $n\to\infty$. Since $S_T\geqslant \mu_n\mathbbm{1}$, Theorem \ref{thm:boundsSTboundT} implies
 \[
  \langle v,T v\rangle\;\geqslant\;\mu_n\|v\|^2+\:\mu_n^2\,\big\|(S_F-\mu_n\mathbbm{1})^{-\frac{1}{2}} v\big\|^2
 \]
 for every $v\in\mathcal{D}(T)$, whence
 \[
  \limsup_{n\to\infty}\big\|(S_F-\mu_n\mathbbm{1})^{-\frac{1}{2}} v\big\|^2\;<\;+\infty\,.
 \]

 In fact, for each $v$ the sequence of square norms $\|(S_F-\mu_n\mathbbm{1})^{-\frac{1}{2}} v\|^2$ is monotone increasing. For, if $\mathfrak{m}(S)>\mu'>\mu$, then
 \[
  \begin{split}
   \big\|(S_F-\mu'\mathbbm{1})^{-\frac{1}{2}} v\big\|^2&-\big\|(S_F-\mu\mathbbm{1})^{-\frac{1}{2}} v\big\|^2 \\
   &=\;\int_{[\mathfrak{m}(S),+\infty)}\Big(\frac{1}{\lambda-\mu'}-\frac{1}{\lambda-\mu}\Big)\,\ud\nu_{v}(\lambda)\;\geqslant\;0\,,
  \end{split}
 \]
 where $\nu_v$ is the scalar spectral measure of the self-adjoint operator $S_F$ relative to the vector $v$. Therefore,
 \[
  \exists\lim_{n\to\infty}\big\|(S_F-\mu_n\mathbbm{1})^{-\frac{1}{2}} v\big\|^2\;<\;+\infty\,.
 \]

 As the latter conclusion is tantamount as
 \[
  \exists\lim_{n\to\infty}\int_{[\mathfrak{m}(S),+\infty)}\frac{1}{\lambda-\mu_n}\,\ud\nu_{v}(\lambda)\;<\;+\infty\,,
 \]
 then by monotone convergence the function $\lambda\mapsto(\lambda-\mathfrak{m}(S))^{-1}$ is $\nu_v$-summable. Thus, $\|(S_F-\mathfrak{m}(S)\mathbbm{1})^{-\frac{1}{2}} v\|^2<+\infty$, whence $v\in\mathrm{ran}(S_F-\mathfrak{m}(S)\mathbbm{1})^{\frac{1}{2}}$.

 On the other hand, since by assumption $S_T$ is a self-adjoint extension of $S$ distinct from $S_F$, then by definition of extension parameter $T$ one has that $\mathcal{D}(T)$ is a non-trivial subspace of $\ker S^*$. Summarising,
 \[
  \mathcal{D}(T)\;\subset\;\mathrm{ran}(S_F-\mathfrak{m}(S)\mathbbm{1})^{\frac{1}{2}}\cap\ker S^*
 \]
 and therefore $\mathrm{ran}(S_F-\mathfrak{m}(S)\mathbbm{1})^{\frac{1}{2}}\cap\ker S^*$ is non-trivial. 
\end{proof}

\begin{corollary}\label{cor:cor_of_nec_statement}
 As a consequence of the necessity statement of Theorem \ref{thm:main1}, for each self-adjoint extension $S_T$ of $S$ with $\mathfrak{m}(S_T)=\mathfrak{m}(S)$ one has
 \begin{equation}\label{eq:coreq}
 \begin{split}
  \mathcal{D}(T)\;&\subset\;\mathrm{ran}(S_F-\mathfrak{m}(S)\mathbbm{1})^{\frac{1}{2}}\cap\ker S^*\qquad\textrm{and} \\
  \langle v,Tv\rangle \;&\geqslant\; \mathfrak{m}(S)\|v\|^2+\mathfrak{m}(S)^2\big\| (S_F-\mathfrak{m}(S)\mathbbm{1})^{-\frac{1}{2}}v\big\|^2\qquad\forall v\in\mathcal{D}(T)\,.
 \end{split}
 \end{equation}
\end{corollary}

\begin{proof}
 The inclusion for $\mathcal{D}(T)$ was already proved. Next, as a follow-up of the reasoning of the previous proof, let us observe that for each $v\in\mathcal{D}(T)$ one has 
 \[
  \lim_{n\to\infty}\big\|(S_F-\mu_n\mathbbm{1})^{-\frac{1}{2}} v\big\|^2\;=\;\big\|(S_F-\mathfrak{m}(S)\mathbbm{1})^{-\frac{1}{2}} v\big\|^2\,.
 \]
 Indeed,
  \[
  \begin{split}
   \big\|(S_F-&\mathfrak{m}(S)\mathbbm{1})^{-\frac{1}{2}} v\big\|^2-\big\|(S_F-\mu_n\mathbbm{1})^{-\frac{1}{2}} v\big\|^2 \\
   &=\;\int_{[\mathfrak{m}(S),+\infty)}\Big(\frac{1}{\lambda-\mathfrak{m}(S)}-\frac{1}{\lambda-\mu_n}\Big)\,\ud\nu_{v}(\lambda)\;\xrightarrow[]{n\to\infty}\;0
  \end{split}
 \]
 by dominated convergence. Therefore, one can take the limit $n\to\infty$ in the inequality
  \[
  \langle v,T v\rangle\;\geqslant\;\mu_n\|v\|^2+\:\mu_n^2\,\big\|(S_F-\mu_n\mathbbm{1})^{-\frac{1}{2}} v\big\|^2
 \]
 thus obtaining the second line of \eqref{eq:coreq}. 
\end{proof}

\begin{proof}[Proof of Theorem \ref{thm:main2} and of Theorem \ref{thm:main1}, sufficiency part]~

(i) The fact that \eqref{eq:defq} defines a symmetric quadratic form with strictly positive lower bound is obvious. As for $q$ being closed, let us show that if $(v_n)_{n\in\mathbb{N}}$ is a sequence in $\mathcal{D}[q]$ with $v_n\to v$ and $q[v_n-v_m]\to 0$ as $n,m\to\infty$, then $v\in\mathcal{D}[q]$ and $q[v_n-v]\to 0$. This is indeed equivalent to saying that $q$ is closed (see, e.g., \cite[Prop.~10.1]{schmu_unbdd_sa}). Now, the above assumption on $(v_n)_{n\in\mathbb{N}}$ implies
\[
 \begin{split}
  v_n\;&\to\; v \\
  (S_F-\mathfrak{m}(S)\mathbbm{1})^{-\frac{1}{2}}v_n&\to\; u
 \end{split}
\]
for some $v\in\cH$. As $(S_F-\mathfrak{m}(S)\mathbbm{1})^{-\frac{1}{2}}$ is self-adjoint and hence closed, this implies
\[
 \begin{split}
  v\;&\in\;\mathcal{D}( (S_F-\mathfrak{m}(S)\mathbbm{1})^{-\frac{1}{2}})\;=\;\mathrm{ran}( (S_F-\mathfrak{m}(S)\mathbbm{1})^{\frac{1}{2}}) \\
  u\;&=\;(S_F-\mathfrak{m}(S)\mathbbm{1})^{-\frac{1}{2}}v\,.
 \end{split}
\]
A first conclusion, since $\ker S^*$ is closed in $\cH$ and hence $v\in\ker S^*$ as well, is that $v\in \mathrm{ran}( (S_F-\mathfrak{m}(S)\mathbbm{1})^{\frac{1}{2}})\cap\ker S^*=\mathcal{D}[q]$. As further conclusion, since $v_n\to v$ and  $(S_F-\mathfrak{m}(S)\mathbbm{1})^{-\frac{1}{2}}v_n\to (S_F-\mathfrak{m}(S)\mathbbm{1})^{-\frac{1}{2}}v$ in $\cH$, one has $q[v_n-v]\to 0$. Part (i) of Theorem \ref{thm:main2} is thus proved.

(ii) As $q$ is densely defined in the Hilbert subspace $\overline{\mathcal{D}[q]}$, and it is symmetric, closed, and semi-bounded from below, then $q$ uniquely identifies a self-adjoint operator $T_q$ on $\overline{\mathcal{D}[q]}$ defined by
\[
 \begin{split}
 \mathcal{D}(T_q)\;&:=\;\big\{ v\in \mathcal{D}[q]\,|\,\exists z_v\in\cH\textrm{ with } \langle u,z_v\rangle=q[u,v]\;\forall u\in\mathcal{D}[q]\big\} \\
 T_qv\;&:=\;z_v
 \end{split}
\]
(see, e.g., \cite[Theorem 10.7]{schmu_unbdd_sa}). Since $\mathcal{D}[q]\subset\ker S^*$ and $\ker S^*$ is closed in $\cH$, then $\overline{\mathcal{D}[q]}\subset\ker S^*$, thus proving that $T_q\in\mathcal{S}(\ker S^*)$. This establishes part (ii) of Theorem \ref{thm:main2}.


(iii) Let $T\in\mathcal{S}(\ker S^*)$ with $T\geqslant T_q$. This means that $\mathcal{D}(T)\subset\mathcal{D}(T_q)$ and 
\[
  \langle v,Tv\rangle\;\geqslant\;\langle v,T_q v\rangle \;=\;q[v]\;=\;\mathfrak{m}(S)\|v\|^2+\mathfrak{m}(S)^2\big\| (S_F-\mathfrak{m}(S)\mathbbm{1})^{-\frac{1}{2}}v\big\|^2
\]
for every $v\in\mathcal{D}(T)$. Consider now an arbitrary $\mu<\mathfrak{m}(S)$. With the very same argument used in the proof of the necessity part of Theorem \ref{thm:main2} one sees that 
\[
 \big\| (S_F-\mathfrak{m}(S)\mathbbm{1})^{-\frac{1}{2}}v\big\|^2\;>\;\big\| (S_F-\mu\mathbbm{1})^{-\frac{1}{2}}v\big\|^2\,,
\]
whence 
\[
  \langle v,Tv\rangle\;>\;\mu\|v\|^2+\mu^2\big\| (S_F-\mu\mathbbm{1})^{-\frac{1}{2}}v\big\|^2
\]
for all $v\in\mathcal{D}(T)$. Owing to Theorem \ref{thm:boundsSTboundT}, the self-adjoint extension $S_T$ of $S$ parametrised by the considered $T$ is such that $S_T\geqslant\mu\mathbbm{1}$. By the arbitrariness of $\mu$, one concludes that $S_T\geqslant\mathfrak{m}(S)\mathbbm{1}$, whence $\mathfrak{m}(S_T)=\mathfrak{m}(S)$. Unless ``$T=\infty$'' (in the sense $\mathcal{D}[T]=\{0\}$), all other choices for $T$ identifies non-Friedrichs extensions. 
This completes the proof of part (iii) of Theorem  \ref{thm:main2}. At the same time, this proves that assumption \eqref{eq:maincond} in Theorem \ref{thm:main1} allows one to construct non-Friedrichs extensions with the same Friedrichs lower bound. Thus also the sufficiency statement of Theorem \ref{thm:main1} is established.

(iv) Last, let $S_T$ be a self-adjoint extension of $S$ with $\mathfrak{m}(S_T)=\mathfrak{m}(S)$. 
The necessity statement of Theorem \ref{thm:main1} implies that the intersection $\mathrm{ran}(S_F-\mathfrak{m}(S)\mathbbm{1})^{\frac{1}{2}}\cap\ker S^*$ is non-trivial, so one can define the form $q$ and the operator $T_q\in\mathcal{S}(\ker S^*)$ as in parts (i) and (ii) of Theorem \ref{thm:main2}. Owing to Corollary \ref{cor:cor_of_nec_statement},
\[
\begin{split}
 \mathcal{D}[T]\;&\subset\;\mathcal{D}(q) \\
 T[v]\;&=\;q[v]\qquad\forall v\in\mathcal{D}[T]\,.
\end{split}
\]
This means precisely that $T\geqslant T_q$.
\end{proof}

\section{Applications}\label{sec:applications}

Let us discuss now a few instructive examples of application of Theorems \ref{thm:main1}-\ref{thm:main2}.

\subsection{Schr\"{o}dinger quantum particle on an interval}\label{sec:example1}~

Let us revisit in more systematic terms the example presented in Sect.~\ref{sec:motivation}. The operator $S$ has deficiency index equal to 2, and explicitly
\begin{equation}
 \ker S^*\;=\;\mathrm{span}\{\mathbf{1},x\}\,.
\end{equation}

The operator $S_F-\pi^2\mathbbm{1}$ fails to be invertible on the whole $\cH=L^2(0,1)$ because it has a non-trivial kernel:
\begin{equation}\label{eq:kerranSF01}
 \begin{split}
  \ker (S_F-\pi^2\mathbbm{1})\;&=\;\mathrm{span}\{\sin\pi x\} \\
  \mathrm{ran}(S_F-\pi^2\mathbbm{1})\;&=\;\mathrm{span}\{\sin n\pi x\,|\,n\in\mathbb{N},n\geqslant 2\}\,.
 \end{split}
\end{equation}
As $(S_F-\pi^2\mathbbm{1})$ is diagonalised as above over an orthonormal basis of eigenfunctions, its  powers 
$(S_F-\pi^2\mathbbm{1})^{\delta}$ and $(S_F-\pi^2\mathbbm{1})^{-\delta}$, with $\delta>0$, are qualified by their action on the same basis of eigenfunctions, with eigenvalues given by the corresponding powers of the eigenvalues of $(S_F-\pi^2\mathbbm{1})$; the negative powers are clearly only defined on the Hilbert subspace $\overline{\mathrm{ran}(S_F-\pi^2\mathbbm{1})}=\{\sin\pi x\}^\perp$. Therefore,
\begin{equation}\label{eq:kerranSF01-bis}
 \mathrm{ran}(S_F-\pi^2\mathbbm{1})^{\frac{1}{2}}\;=\;\mathrm{ran}(S_F-\pi^2\mathbbm{1})\;=\;\mathrm{span}\{\sin n\pi x\,|\,n\in\mathbb{N},n\geqslant 2\}\,.
\end{equation}

\begin{lemma}\label{lem:VandactionSF}
One has
\begin{equation}\label{eq:EX1V}
 V\;:=\;\mathrm{ran}(S_F-\pi^2\mathbbm{1})^{\frac{1}{2}}\cap \ker S^*\;=\;\mathrm{span}\{\mathbf{1}-2x\}
\end{equation}
and 
\begin{equation}\label{eq:EX1SF} 
 (S_F-\pi^2\mathbbm{1})^{-1}(\mathbf{1}-2x)\;=\;\pi^{-2}\big(\cos\pi x-\mathbf{1}+2x\big)\,.
\end{equation}
\end{lemma}

\begin{proof}
 In order for a generic element $a\mathbf{1}+bx\in\ker S^*$, with $a,b\in\mathbb{C}$, to belong to $\mathrm{ran}(S_F-\pi^2\mathbbm{1})^{\frac{1}{2}}$, owing to \eqref{eq:kerranSF01}-\eqref{eq:kerranSF01-bis} it must be
 \[
  0\;=\;\int_0^1(a+bx)\,\sin\pi x\,\ud x\;=\;\pi^{-1}(2a+b)\,,
 \]
 whence $b=-2a$. Thus $g\in V$ implies $g=a(\mathbf{1}-2x)$ for some $a\in\mathbb{C}$. Next, one has to check that $\mathbf{1}-2x\in \mathrm{ran}(S_F-\pi^2\mathbbm{1})^{\frac{1}{2}}$. This is the same as $\mathbf{1}-2x\in \mathrm{ran}(S_F-\pi^2\mathbbm{1})$\,, that is, 
 \[
  \mathbf{1}-2x\;=\;(S_F-\pi^2\mathbbm{1})u\qquad\textrm{for some }u\in\mathcal{D}(S_F)\,.
 \]
 This is equivalent to saying that $u$ is the \emph{minimal norm solution} to the boundary value problem
 \[
  \begin{cases}
   \;-u''-\pi^2 u\;=\;1-2x \\
   \;u(0)=0=u(1)\,.
  \end{cases}
 \]
 By standard ODE methods one finds that the general solution is
 \[
  u_{\mathrm{gen}}(x)\;=\;\pi^{-2}\big(\cos\pi x-1+2x\big)+B\sin\pi x\,,\qquad B\in\mathbb{C}\,,
 \]
 thus the minimal norm solution is the one with $B=0$. This proves that the function $u_\circ:=\pi^{-2}\big(\cos\pi x-\mathbf{1}+2x\big)\in\mathcal{D}(S_F)$ satisfies $(S_F-\pi^2\mathbbm{1})u_\circ=\mathbf{1}-2x$, thus completing the proof of \eqref{eq:EX1V} and \eqref{eq:EX1SF}.
\end{proof}

As the intersection space \eqref{eq:EX1V} is non-trivial, Theorem \ref{thm:main1} ensures that $S$ admits non-Friedrichs self-adjoint extensions with the same Friedrichs lower bound. This is consistent with what discussed in the introduction: $\mathfrak{m}(S_F)=\mathfrak{m}(S_A)$, Friedrichs and anti-periodic extension have the same lower bound.

It is instructive to apply the constructive recipe of Theorem \ref{thm:main2} so as to identify all such extensions. With the notation therein, 
\begin{equation}
 \mathcal{D}[q]\;=\;\mathcal{D}(T_q)\;=\;V\;=\;\mathrm{span}\{\mathbf{1}-2x\}\,,
\end{equation}
thus $T_q$ is an operator of multiplication by some real number $t_q$,
\begin{equation}
 T_q(\mathbf{1}-2x)\;=\;t_q(\mathbf{1}-2x)\,.
\end{equation}
Since
\begin{equation*}
\begin{split}
 \langle (\mathbf{1}-2x), T_q (\mathbf{1}-2x)\rangle\;&=\;\pi^2\|\mathbf{1}-2x\|_2^2+\pi^4\big\| (S_F-\pi^2\mathbbm{1})^{-\frac{1}{2}}(\mathbf{1}-2x)\big\|_2^2 \\
 &=\;\pi^2\|\mathbf{1}-2x\|_2^2+\pi^4\langle (\mathbf{1}-2x),(S_F-\pi^2\mathbbm{1})^{-1}(\mathbf{1}-2x)\rangle \\
 &=\;\pi^2\|\mathbf{1}-2x\|_2^2+\pi^2\langle (\mathbf{1}-2x),\cos x-(\mathbf{1}-2x))\rangle \\
 &=\;\pi^2\langle (\mathbf{1}-2x),\cos x\rangle \\
 &=\;4\;=\;12\,\|\mathbf{1}-2x\|_2^2
 \end{split}
\end{equation*}
(having used \eqref{eq:EX1SF} in the third step and $\|\mathbf{1}-2x\|_2^2=\frac{1}{3}$ in the last step),
then necessarily $t_q=12$.

Theorem \ref{thm:main2}, in parts (iii) and (iv), then states that the self-adjoint extensions $S_T$ of $S$ with $\mathfrak{m}(S_T)=\mathfrak{m}(S_F)$ are those labelled by self-adjoint operators $T$ with $T\geqslant T_q$. Such $T$'s, apart from the one parametrising the Friedrichs extension, are therefore such that
\begin{equation}\label{eq:goodTs}
 \begin{split}
  & \mathcal{D}(T)\;=\;V\;=\;\mathrm{span}\{\mathbf{1}-2x\} \\
  & \textrm{$T$ is the multiplication by some $t\geqslant 12$}\,.
 \end{split}
\end{equation}
Keeping into account, as is immediate to check, that 
\begin{equation}
 W\;:=\;V^\perp\cap\ker S^*\;=\;\mathrm{span}\{\mathbf{1}\}\,,
\end{equation}
the extension $S_T$ for each $T$ satisfying \eqref{eq:goodTs} has domain given by formula \eqref{eq:ST} of Theorem \ref{thm:KVB-parametrisation}, that is,
\begin{equation}\label{eq:STforlaplace01}
\begin{split}
\mathcal{D}(S_T)\;&=\;\left\{f+S_F^{-1}(Tv+w)+v\left|\!\!
\begin{array}{c}
f\in\mathcal{D}(\overline{S})\,, \\
v\in V\,,\;w\in W
\end{array}\!\!
\right.\right\} \\
&=\;\left\{f+S_F^{-1}(t\alpha(\mathbf{1}-2x)+\beta\mathbf{1})+\alpha(\mathbf{1}-2x)\left|\!\!
\begin{array}{c}
f\in H^2_0(0,1) \\
\alpha,\beta\in\mathbb{C}
\end{array}\!\!
\right.\right\} .
\end{split}
\end{equation}
The action of the everywhere defined and bounded operator $S_F^{-1}$ on the subspace $\ker S^*=\mathrm{span}\{\mathbf{1},x\}$ is easily computed by solving a boundary value problem completely analogous to the one considered in the proof of Lemma \ref{lem:VandactionSF}. The result (as found, e.g., in \cite[Eq.~(91)]{GMO-KVB2017}) is
\begin{equation}
 S_F^{-1}(a\mathbf{1}+bx)\;=\;\Big(\frac{a}{2}+\frac{b}{6}\Big)x-\frac{a}{2}x^2-\frac{b}{6}x^3\,.
\end{equation}
Thus,
\begin{equation}\label{eq:finalwithKVBparam}
 \mathcal{D}(S_T)\,=\,\left\{f+\alpha\mathbf{1}+\Big(\frac{t\alpha+3\beta}{6}-2\alpha\Big)x-\frac{t\alpha+\beta}{2}x^2+\frac{t\alpha}{3}x^3\left|\!\!
\begin{array}{c}
f\in H^2_0(0,1) \\
\alpha,\beta\in\mathbb{C}
\end{array}\!\!
\right.\right\}. \!\!\!\!\!\!\!\!\!\!\!\!
\end{equation}

Formula \eqref{eq:finalwithKVBparam}, for each fixed $t\geqslant 12$, identifies those self-adjoint extensions of $S$ different from the Friedrichs extension, but with the same lower bound. In order to identify the boundary condition of self-adjointness satisfied by a generic element $g\in\mathcal{D}(S_T)$ for each extension of type \eqref{eq:finalwithKVBparam}, we compute the boundary values
\begin{equation}\label{eq:bc-1}
 \begin{array}{lcl}
  g(0)\;=\;\alpha & & g'(0)\;=\;\frac{t\alpha}{6}+\frac{\beta}{2}-2\alpha \\
  g(1)\;=\;-\alpha & & g'(1)\;=\;\frac{t\alpha}{6}-\frac{\beta}{2}-2\alpha
 \end{array}
\end{equation}
and re-write
\begin{equation}\label{eq:bc-2}
 \begin{split}
  g(0)+g(1)\;&=\;0 \\
  g'(0)+g'(1)\;&=\;{\textstyle\frac{1}{3}}(t-12)\,g(0)\,.
 \end{split}
\end{equation}
It was indeed convenient to cast \eqref{eq:bc-1} in the form \eqref{eq:bc-2} because the latter can be more easily matched with the general conditions of self-adjointness of the extensions of $S$, as we shall now do.

We refer to the following very standard result, obtained for example by exploiting Theorem \ref{thm:KVB-parametrisation} for all possible extension parameters (see, e.g., \cite[Example 14.10]{schmu_unbdd_sa}), or equivalently by means of the alternative extension scheme a la von Neumann applied to $S$ (see, e.g., \cite[Sect.~6.2.3.1]{GTV-2012}).

\begin{proposition}
 The family of self-adjoint extensions on $L^2(0,1)$ of the operator $S$ defined in \eqref{eq:Slaplace01} consists of restrictions of $S^*$, and hence of operators of the form $-\frac{\ud^2}{\ud x^2}$, to domains of $H^2(0,1)$-functions $g$ satisfying boundary conditions of one of the following four classes:
\begin{equation}\label{bc2}
g'(0)\;=\;b_1 g(0) + c g(1)\,,\qquad g'(1)\;=\;-\overline{c} g(0) - b_2 g(1)\,,
\end{equation}
\begin{equation}\label{bc1a}
g'(0)\;=\;b_1 g(0) + \overline{c} g'(1)\,,\qquad g(1)\;=\; c g(0)\,,
\end{equation}
\begin{equation}\label{bc1b}
g'(1)\;=\;-b_1 g(1)\,,\qquad g(0)\;=\;0\,,
\end{equation}
\begin{equation}\label{bc0}
g(0)\;=\;0\;=\;g(1)\,,
\end{equation}
where $c\in\mathbb{C}$ and $b_1,b_2\in\mathbb{R}$ and qualify each extension.
\end{proposition}

\begin{remark}
With reference to the general formula \eqref{eq:ST}, extensions of type \eqref{bc2} correspond to the case in which $\dim\mathcal{D}(T)=2$, extensions of type \eqref{bc1a} or \eqref{bc1b} correspond to $\dim\mathcal{D}(T)=1$, and the extension of type \eqref{bc0} is the Friedrichs extension, $\dim\mathcal{D}(T)=0$. 
\end{remark}

By direct comparison between \eqref{eq:bc-2} and \eqref{bc2}-\eqref{bc0} we see that \eqref{eq:bc-2} can only be of type \eqref{bc1a} with
\begin{equation}\label{eq:choiceb1c}
 b_1\;=\;{\textstyle\frac{1}{3}}(t-12)\qquad\textrm{ and }\qquad c\;=\;-1\,.
\end{equation}
We have thus proved the following.

\begin{proposition}\label{prop:ExampleDD}
 The non-Friedrichs self-adjoint extensions on $L^2(0,1)$ of the operator $S$ defined in \eqref{eq:Slaplace01} which preserve the Friedrichs lower bound $\mathfrak{m}(S)=\pi^2$ are all those operators acting as $-\frac{\ud^2}{\ud x^2}$ on $H^2(0,1)$-functions $g$ with boundary condition
 \begin{equation}\label{eq:bc-3}
 \begin{split}
  g(0)+g(1)\;&=\;0 \\
  g'(0)+g'(1)\;&=\;b\,g(0)
 \end{split}
\end{equation}
 for fixed $b\geqslant 0$. Each $b$ qualifies one of such extensions, with a one-to-one correspondence. Such extensions are ordered with increasing $b$. The choice $b=0$ corresponds to anti-periodic boundary conditions.
\end{proposition}

The application of Theorems \ref{thm:main1} and \ref{thm:main2} thus allowed for a fast identification of all non-Friedrichs extensions with Friedrichs lower bound of the minimally defined Laplacian on $[0,1]$, which would have otherwise required a tedious computation, by means of \eqref{bc2}-\eqref{bc0}, of all the discrete spectra of the various extensions, in order to select those with bottom equal to $\pi^2$. 

For completeness, here is how the direct check would have proceeded. Let us limit the analysis to the eigenvalue problem for a generic self-adjoint extension of type \eqref{bc1a} with the choice \eqref{eq:choiceb1c}, namely
\begin{equation}\label{eq:ODEproblemDD}
 \begin{cases}
 \; -g''\;=\;\lambda g \\
 \;  g(0)+g(1)\;=\;0 \\
 \;  g'(0)+g'(1)\;=\;{\textstyle\frac{1}{3}}(t-12)\,g(0) \\
 \; (\lambda\in\mathbb{R}\,,\;g\in H^2(0,1))
 \end{cases}
\end{equation}
for fixed $t\in\mathbb{R}$. $g$ must be of the form $g(x)=A\cos\sqrt{\lambda}x+B\sin\sqrt{\lambda}x$, $A,B\in\mathbb{C}$, and for sure the pairs $(g,\lambda)$ with
\begin{equation}\label{eq:firstpartofspectrum}
 g(x)\;=\;\sin((2n+1) \pi x)\,,\qquad \lambda =(2n+1)^2 \pi^2\,,\qquad n \in \mathbb{N}_0 
\end{equation}
solve \eqref{eq:ODEproblemDD}, showing that all such extensions have the eigenvalues $(2n+1)^2 \pi^2$, $n \in \mathbb{N}_0$, in common. The remaining (i.e., non-$\sin$-only) solutions to \eqref{eq:ODEproblemDD} are obtained imposing $B\neq 0$, and it is then simple to conclude that the admissible $\lambda$'s are the ($t$-dependent) roots of
\begin{equation}\label{eq:rooteq}
 F(\lambda)\;=\;t\,,\qquad\textrm{where}\qquad F(\lambda) \;:=\; 12 - 6\sqrt{\lambda}\,\textstyle \frac{\,1+\cos\sqrt{\lambda}}{\sin\sqrt\lambda}
\end{equation}
(and understanding the above trigonometric functions as hyperbolic functions when $\lambda<0$).
As $F(\lambda)$ increases with $\lambda$ in all intervals in which it is defined, and $F(\pi^2)=12$, one deduces that only for $t\geqslant 12$ the admissible $\lambda$'s selected by \eqref{eq:rooteq} satisfy $\lambda\geqslant \pi^2$ (see Figure \ref{fig:evs}). The spectrum thus determined from \eqref{eq:firstpartofspectrum} and \eqref{eq:rooteq} indeed confirms, by direct inspection, what found in Prop.~\ref{prop:ExampleDD} by means of our Theorem \ref{thm:main2}.

\begin{figure}[h!]
\includegraphics[width=0.45\textwidth]{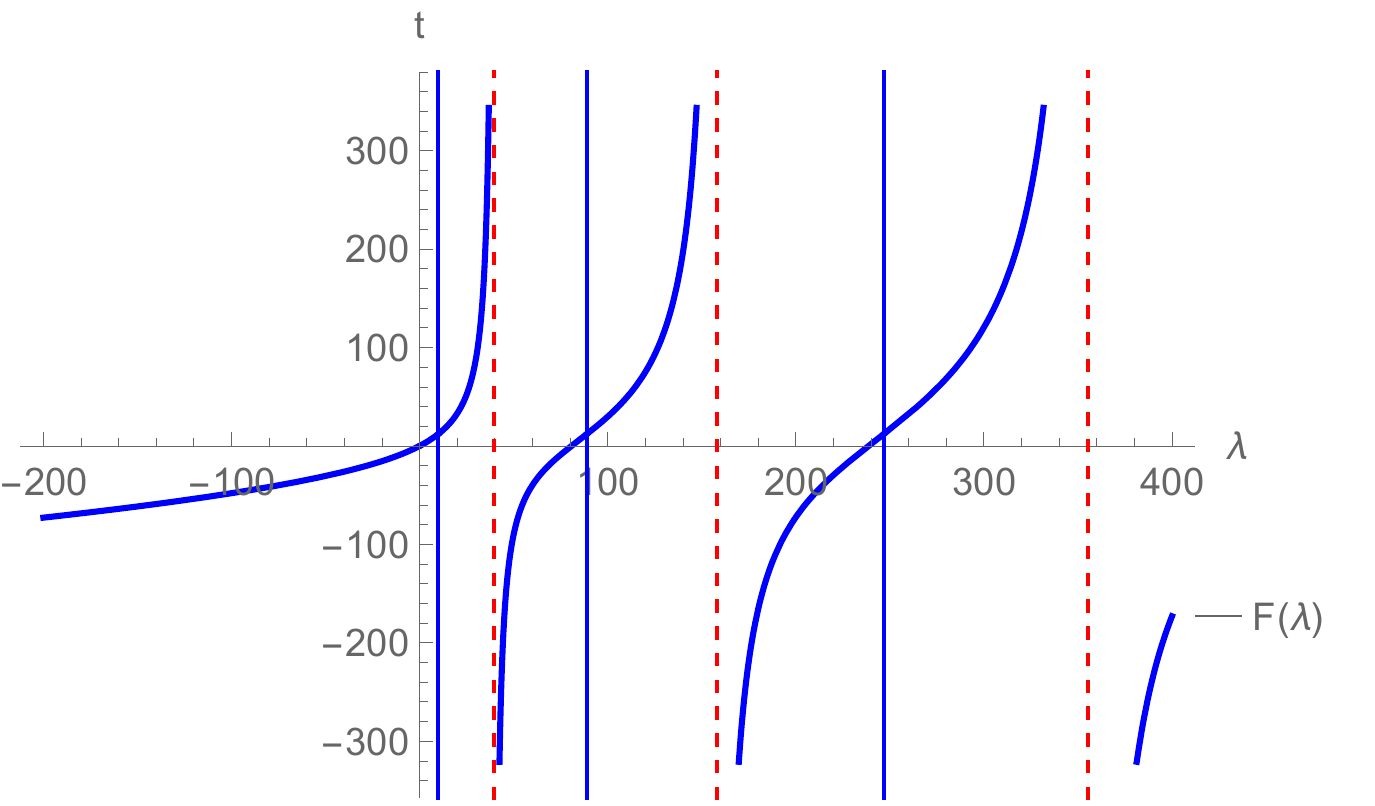} \quad \includegraphics[width=0.45\textwidth]{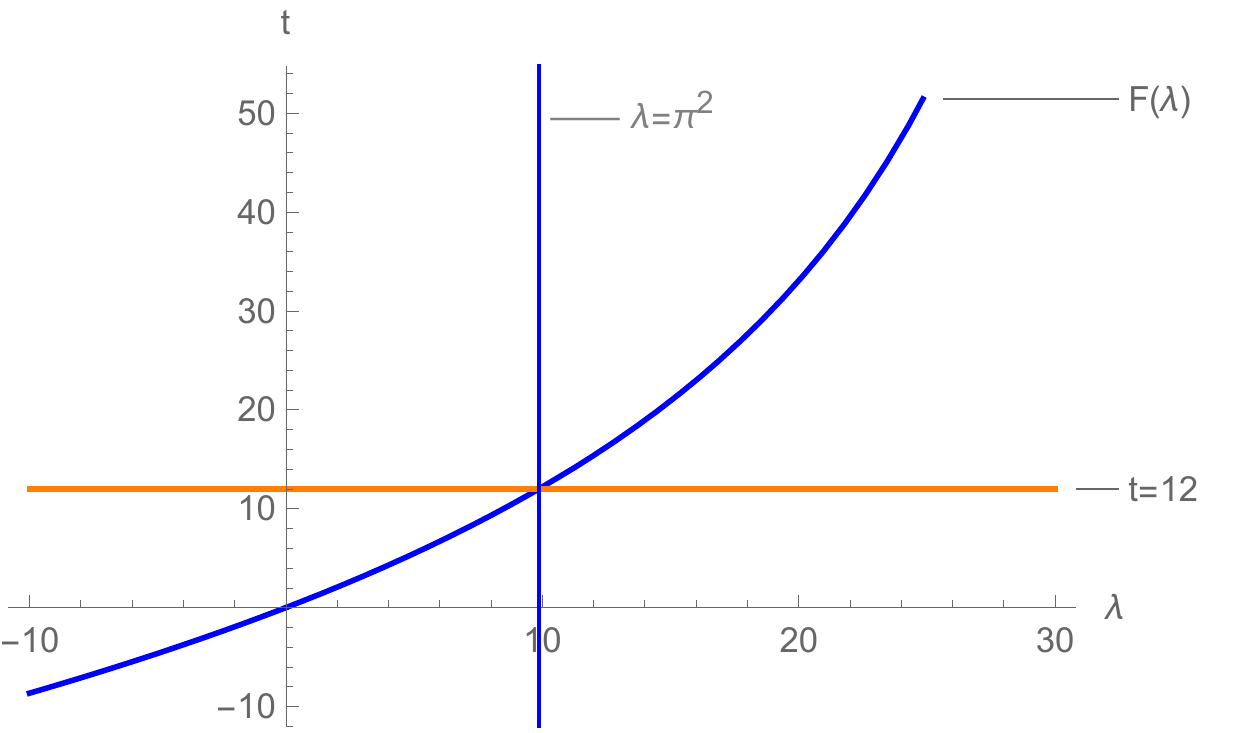} \label{fig:evs}

\caption{Left: plot of $F(\lambda)$ defined in \eqref{eq:rooteq} (blue curves) as compared to the eigenvalues of type \eqref{eq:firstpartofspectrum} (vertical blue lines) and the Friedrichs eigenvalues (dashed red lines). Right: magnification of the first positive interval of definition of $F(\lambda)$. Eigenvalues determined by \eqref{eq:firstpartofspectrum} and \eqref{eq:rooteq} correspond to the intersections of the blue curves with the horizontal lines at level $t$.}
\end{figure}

\subsection{Schr\"{o}dinger quantum particle in $\mathbb{R}^3$ with point interaction}\label{sec:example2}~

This is an example with deficiency index equal to 1. With respect to the Hilbert space $\cH=L^2(\mathbb{R}^3)$ we consider the operator
\begin{equation}
 \mathcal{D}(\widetilde{S})\;=\;C^\infty_0(\mathbb{R}^3\setminus\{0\})\,,\qquad \widetilde{S}\;=\;-\Delta\,.
\end{equation}
$\widetilde{S}$ is densely defined and symmetric, with $\mathfrak{m}(\widetilde{S})=0$.

The self-adjoint extensions of $\widetilde{S}$ are Hamiltonians for a quantum particle in three dimensions subject to a point interaction supported at $x=0$, and they are very well studied and understood.

\begin{theorem}\label{thm:Deltaalpha}\emph{[See, e.g., \cite[Chapt.~I.1]{albeverio-solvable}.]}
\begin{itemize}
 \item[(i)] $\widetilde{S}$ has unit deficiency index. The Friedrichs extension of $\widetilde{S}$ is the self-adjoint (negative) Laplacian on $L^2(0,1)$ with domain $H^2(0,1)$. All other self-adjoint extensions of $\widetilde{S}$ form the family $\{-\Delta_{\alpha}\,|\,\alpha\in\mathbb{R}\}$, where
 \begin{equation}
  \begin{split}
   \mathcal{D}(-\Delta_\alpha)\;&=\;\left\{ g=\phi+\frac{\phi(0)}{\alpha+\frac{1}{4\pi}} G_1 \,\Big|\,\phi\in H^2(\mathbb{R}^3)\right\} \\
   (-\Delta_\alpha+\mathbbm{1})g\;&=\;(-\Delta+\mathbbm{1})\phi
  \end{split}
 \end{equation}
 and
 \begin{equation}
  G_1\;:=\;(2\pi)^{\frac{3}{2}}\frac{\,e^{-|x|}}{\,4\pi|x|\,}\,.
 \end{equation}
 \item[(ii)] For each $\alpha\in\mathbb{R}$,
\begin{equation}
\sigma_{\mathrm{ess}}(-\Delta_\alpha)\;=\;\sigma_{\mathrm{ac}}(-\Delta_\alpha)\;=\;[0,+\infty)\,,\qquad \sigma_{\mathrm{sc}}(-\Delta_\alpha)\;=\;\emptyset\,,
\end{equation}
and
\begin{equation}
\sigma_{\mathrm{p}}(-\Delta_\alpha)\;=\;
\begin{cases}
\qquad \emptyset & \textrm{if }\alpha\in[0,+\infty] \\
\{-(4\pi\alpha)^2\} & \textrm{if }\alpha\in(-\infty,0)\,.
\end{cases}
\end{equation}
The negative eigenvalue $-(4\pi\alpha)^2$, when it exists, is simple and the corresponding eigenfunction is $|x|^{-1}e^{4\pi\alpha|x|}$.
\end{itemize}
\end{theorem}

We see from Theorem \ref{thm:Deltaalpha} that $\widetilde{S}$ admits a collection of non-Friedrichs extensions with Friedrichs lower bound, and precisely
\begin{equation}\label{eq:collectionboundzero}
 \mathfrak{m}(-\Delta_\alpha)\;=\;0\;=\;\mathfrak{m}(\widetilde{S})\qquad\forall\alpha\geqslant 0\,.
\end{equation}

In order to recover such a conclusion from the abstract setting of Sect.~\ref{sec:abstractresults}, let us consider
\begin{equation}
 S\;:=\;\widetilde{S}+\mathbf{1}\,.
\end{equation}
Clearly, $\mathfrak{m}(S)=1$. The self-adjoint extensions of $\widetilde{S}$ and of $S$ then only differ by a trivial shift. As we intend to analyse the extensions of $S$ within the extension scheme of Theorem \ref{thm:KVB-parametrisation}, rather than using von Neumann's extension theorem as in  \cite{albeverio-solvable}, let us follow closely the discussion made in \cite[Sect.~3]{MO-2016}, were indeed the Kre{\u\i}n-Vi\v{s}ik-Birman scheme was employed.

We shell denote, respectively, by $\;\widehat{}\;$ and ${}_{^{\textrm{\Huge$\check{\,}$\normalsize}}}$ the Fourier and inverse Fourier transform $L^2(\mathbb{R}^3,\ud x)\to L^2(\mathbb{R}^3,\ud p)$ with the convention
\[
 \widehat{f}(p)\;=\;\frac{1}{\;(2\pi)^{\frac{3}{2}}}\int_{\mathbb{R}^3}e^{-\ii\,p\cdot x}f(x)\,\ud x\,.
\]
In particular,
\begin{equation}
 G_1\;=\;(2\pi)^{\frac{3}{2}}\frac{\,e^{-|x|}}{\,4\pi|x|\,}\;=\;\Big(\frac{1}{p^2+1}\Big){\textrm{\Huge$\check{\,}$\normalsize}}\,.
\end{equation}

%
%

It is possible to prove the following.

\begin{theorem}\label{thm:fromMO2016}\emph{\cite[Sect.~3]{MO-2016}}~
 \begin{itemize}
  \item[(i)] $S$ has deficiency space
  \begin{equation}
   \ker S^*\;=\;
   \mathrm{span}\{G_1\}\,.
  \end{equation}
  \item[(ii)] The Friedrichs extension of $S$ is the operator
  \begin{equation}
   \begin{split}
    \mathcal{D}(S_F)\;&=\;\left\{g\in L^2(\mathbb{R}^3)\left|
    \begin{array}{c}
      \widehat{g}=\widehat{f}+(p^2+1)^{-1}\eta \\
      f\in\mathcal{D}(\overline{S})\,,\,\eta\in\mathbb{C}
    \end{array}
    \!\!\right.\right\} \\
    \widehat{S_F g}\;&=\;(p^2+1)\widehat{g}\,.
   \end{split}
  \end{equation}
  \item[(iii)] All other self-adjoint extensions of $S$ are of the form $S_t$ for some $t\in\mathbb{R}$, where
    \begin{equation}\label{eq:Stclassification}
   \begin{split}
    \mathcal{D}(S_t)\;&=\;\left\{g\in L^2(\mathbb{R}^3)\left|
    \begin{array}{c}
      \widehat{g}=\widehat{f}+(p^2+1)^{-2}t\xi+(p^2+1)^{-1}\xi \\
      f\in\mathcal{D}(\overline{S})\,,\,\xi\in\mathbb{C}
    \end{array}
    \!\!\right.\right\} \\
    \widehat{S_t g}\;&=\;(p^2+1)\big(\widehat{f}+(p^2+1)^{-2}t\xi \big)\,.
   \end{split}
  \end{equation}
  This is precisely formula \eqref{eq:ST} of Theorem \ref{thm:KVB-parametrisation} specialised to the case where $\ker S^*$ is one-dimensional and $T$ is therefore the operator of multiplication by the real number $t$.
  \item[(iv)] One has 
  \begin{equation}\label{eq:correspformula}
   S_t\;=\;-\Delta_\alpha+\mathbbm{1}\qquad\textrm{for}\qquad \alpha\;=\;\frac{t-2}{8\pi}\,.
   \end{equation}
 \end{itemize}
\end{theorem}

 Clearly, $S_F-\mathfrak{m}(S)\mathbbm{1}=S_F-\mathbbm{1}=\widetilde{S}_F$, the self-adjoint (negative) Laplacian on $L^2(\mathbb{R}^3)$. Therefore, unlike the example discussed in Subsect.~\ref{sec:example1}, 
 \begin{equation}
  \ker (S_F-\mathbbm{1})\;=\;\{0\}\,.
 \end{equation}
  $S_F-\mathbbm{1}$ is then invertible on its range and so are the powers $(S_F-\mathbbm{1})^\delta$, $\delta>0$. On such a space, $(S_F-\mathbbm{1})^{-\delta}$ acts, in Fourier transform, as the multiplication by $|p|^{-2\delta}$.

  The analogue of Lemma \ref{lem:VandactionSF} is now the following.

  \begin{lemma}\label{lem:VandactionSF-2ndEx}
One has
\begin{equation}\label{eq:EX1V-2ndEx}
 V\;:=\;\mathrm{ran}(S_F-\mathbbm{1})^{\frac{1}{2}}\cap \ker S^*\;=\;\mathrm{span}\{G_1\}
\end{equation}
and 
\begin{equation}\label{eq:EX1SF-2ndEx} 
 (S_F-\mathbbm{1})^{-\frac{1}{2}} G_1\;=\;\Big(\frac{1}{\,|p|(p^2+1)\,}\Big){\textrm{\Huge$\check{\,}$\normalsize}}\,.
\end{equation}
\end{lemma}

\begin{proof}
 The fact that $G_1\in\ker S^*$ is stated in Theorem \ref{thm:fromMO2016}(i). As
 \[
  \frac{1}{\,|p|(p^2+1)\,}\;\in\;L^2(\mathbb{R}^3,\ud p)
 \]
 and
 \[
  (S_F-\mathbbm{1})^{\frac{1}{2}}\Big(\frac{1}{\,|p|(p^2+1)\,}\Big){\textrm{\Huge$\check{\,}$\normalsize}}\;=\;\Big(\frac{1}{(p^2+1)}\Big){\textrm{\Huge$\check{\,}$\normalsize}}\;=\;G_1\,,
 \]
 hence $G_1\in \mathrm{ran}(S_F-\mathbbm{1})^{\frac{1}{2}}$. $V$ can be at most one-dimensional, thus \eqref{eq:EX1V-2ndEx} is proved, and so is \eqref{eq:EX1SF-2ndEx} as well. 
\end{proof}

 Owing to Lemma \ref{lem:VandactionSF-2ndEx}, Theorem \ref{thm:main1} is applicable: $S$ admits non-Friedrichs extensions with Friedrichs lower bound, and so does therefore $\widetilde{S}$, consistently with what previously observed in \eqref{eq:collectionboundzero}.

 Furthermore, with the notation of Theorem \ref{thm:main2},
 \begin{equation}
 \mathcal{D}[q]\;=\;\mathcal{D}(T_q)\;=\;V\;=\;\mathrm{span}\{G_1\}\,,
\end{equation}
thus $T_q$ is an operator of multiplication by some real number $t_q$,
\begin{equation}
 T_q \,G_1\;=\;t_q\,G_1\,.
\end{equation}
Since
\begin{equation*}
\begin{split}
 \langle G_1, T_q \,G_1\rangle\;&=\;\|G_1\|_2^2+\big\| (S_F-\mathbbm{1})^{-\frac{1}{2}}G_1\big\|_2^2 \\
 &=\;\Big\|\frac{1}{p^2+1}\Big\|_2^2+\Big\|\frac{1}{\,|p|(p^2+1)}\Big\|_2^2 \\
 &=\;\pi^2+\pi^2 \\
 &=\;2\,\|G_1\|_2^2
 \end{split}
\end{equation*}
(having used \eqref{eq:EX1SF-2ndEx} in the second identity),
then necessarily $t_q=2$.

Theorem \ref{thm:main2}, in parts (iii) and (iv), then states that the self-adjoint extensions $S_T$ of $S$ with $\mathfrak{m}(S_T)=\mathfrak{m}(S_F)$ are those labelled by self-adjoint operators $T$ with $T\geqslant T_q$. Such $T$'s, apart from the one parametrising the Friedrichs extension, are therefore such that
\begin{equation}\label{eq:goodTs-2ndEx}
 \begin{split}
  & \mathcal{D}(T)\;=\;V\;=\;\mathrm{span}\{G_1\} \\
  & \textrm{$T$ is the multiplication by some $t\geqslant 2$}\,.
 \end{split}
\end{equation}
 For what argued in Theorem \ref{thm:fromMO2016}(iii), such extensions are precisely the operators $S_t$ that one reads out from formula \eqref{eq:Stclassification} with $t\geqslant 2$. In turn, the correspondence formula \eqref{eq:correspformula} leads to the conclusion that the self-adjoint extensions of $\widetilde{S}$ with Friedrichs lower bound are precisely those $-\Delta_\alpha$'s with $\alpha\geqslant 0$.

 \subsection{Radial problem in hydrogenoid-like Hamiltonians}~

 It is worth mentioning another example with unit deficiency index, even without working out here the steps through which Theorems \ref{thm:main1} and \ref{thm:main2} are applied, which are in fact completely analogous to the computations of Sect.~\ref{sec:example1} and \ref{sec:example2}.

 For given $\nu\in\mathbb{R}$, let us now consider
 \begin{equation}
   \mathcal{D}(S_\nu)\;=\;C^\infty_0(\mathbb{R}^+)\,,\qquad S_\nu\;=\;-\frac{\ud^2}{\ud x^2}+\frac{\nu}{x}\,,
 \end{equation}
 a densely defined and symmetric operator on the Hilbert space $\cH=L^2(\mathbb{R}^+)$ with lower bound $\mathfrak{m}(S_\nu)=0$. One typical emergence of $S_\nu$ in mathematical physics is as the minimally defined zero-momentum radial operator in the construction of a quantum hydrogenoid Hamiltonian with an additional point interaction at the center of the Coulomb potential: $S_\nu$ is indeed well known and thoroughly studied, and we refer to \cite[Sect.~1.4]{GM-hydrogenoid-2018} and references therein for an updated historical overview.

 Hardy's inequality implies that $S_\nu$ is lower semi-bounded, and in particular obviously
 \begin{equation}
  \mathfrak{m}(S_\nu)\;=\;0\qquad\forall\nu\geqslant 0
 \end{equation}
 (repulsive Coulomb interaction). A standard limit-point limit-circle argument shows that $S_\nu$ has unit deficiency index. Its self-adjoint extensions are studied in the literature by means of various extension schemes, including recently in \cite{GM-hydrogenoid-2018} by means of the general Theorem \ref{thm:KVB-parametrisation} above.

 \begin{theorem}\label{thm:1Dclass}\emph{\cite[Theorems 2 and 4]{GM-hydrogenoid-2018}}.~
\begin{itemize}
\item[(i)] The self-adjoint extensions of $S_\nu$ in $L^2(\mathbb{R}^+)$ form the family $\{S_\nu^{(\alpha)}\,|\,\alpha \in \mathbb{R} \cup \{\infty\}\}$, where $\alpha=\infty$ labels the Friedrichs extension, and 
\begin{equation}\label{eq:dsa}
 \begin{split}
  \mathcal{D}(S_\nu^{(\alpha)})\;&=\;\left\{
  g\in L^2(\mathbb{R}^+)\left|
  \begin{array}{c}
   -g''+\textstyle{\frac{\nu}{r}g}\in L^2(\mathbb{R}^+) \\
   \textrm{and }\;g_1\;=\;4\pi\alpha\, g_0
  \end{array}\!
  \right.\right\} \\
  S_\nu^{(\alpha)}\,g\;&=\;-g''+\frac{\nu}{r}\,g\,,
 \end{split}
\end{equation}
$g_0$ and $g_1$ being the existing limits
\begin{equation}\label{eq:g0g1limits-statements}
 \begin{split}
  g_0\;&:=\;\lim_{r\downarrow 0}g(r) \\
  g_1\;&:=\;\lim_{r\downarrow 0}r^{-1}\big(g(r)-g_0(1+\nu r\ln r)\big)\,.
 \end{split}
\end{equation}
\item[(ii)] For $\nu>0$ one has 
\begin{equation}
\sigma_{\mathrm{ess}}(S_\nu^{(\alpha)})\;=\;\sigma_{\mathrm{ac}}(S_\nu^{(\alpha)})\;=\;[0,+\infty)\,,\qquad \sigma_{\mathrm{sc}}(S_\nu^{(\alpha)})\;=\;\emptyset\,,
\end{equation}
and
\begin{equation}
\sigma_{\mathrm{p}}(S_\nu^{(\alpha)})\;=\;
\begin{cases}
\quad\; \emptyset & \textrm{if }\,\alpha\geqslant\alpha_\nu \\
\{E_+^{(\nu,\alpha)}\} & \textrm{if }\,\alpha<\alpha_\nu\,,
\end{cases}
\end{equation}
where 
\begin{equation}
  \alpha_\nu\;:=\;\frac{\nu}{4\pi}\,(\ln\nu+2\gamma-1)
 \end{equation}
($\gamma\sim 0.577$ being the Euler-Mascheroni constant) and $E_+^{(\nu,\alpha)}$ is the only simple negative root of $\mathfrak{F}_\nu(E)=\alpha$ with
 \begin{equation}\label{eq:Feigenvalues}
 \mathfrak{F}_\nu(E)\;:=\;\frac{\nu}{4\pi}\Big(\psi\big(1+{\textstyle\frac{\nu}{2\sqrt{|E|}}}\big)+ \ln(2 \sqrt{|E|}) +2\gamma - 1 - {\textstyle\frac{\sqrt{|E|}}{\nu}} \Big)
\end{equation}
($\psi(z)=\Gamma'(z)/\Gamma(z)$ being the digamma function).
\end{itemize}
\end{theorem}

 When $\nu>0$ Theorem \ref{thm:1Dclass} thus shows that
 \begin{equation}\label{eq:examplesamebottom}
  \begin{array}{lcl}
   \mathfrak{m}(S_\nu^{(\alpha)})\;=\;E_+^{(\nu,\alpha)}\;<\;0\;=\;\mathfrak{m}(S_\nu) & & \textrm{if }\alpha<\alpha_\nu \\
   \mathfrak{m}(S_\nu^{(\alpha)})\;=\;0\;=\;\mathfrak{m}(S_\nu) & & \textrm{if }\alpha\geqslant\alpha_\nu \,,
  \end{array}
 \end{equation}
 yet another example of the presence of a sub-class of non-Friedrichs extensions with Friedrichs lower bound.

 Also on this example it is easy to test the applicability of our Theorems \ref{thm:main1}-\ref{thm:main2}. As done in Sect.~\ref{sec:example2}, since $\mathfrak{m}(S_\nu)=0$, a positive shift must be performed first. For similar purposes the analysis of the shifted operator
 \begin{equation}\label{eq:oneshift}
  \mathcal{S}_\nu\;:=\;S_\nu+\frac{\nu^2}{4\kappa^2}\mathbbm{1}\qquad (\kappa\in\mathbb{R})
 \end{equation}
 and of its self-adjoint extensions was worked in \cite[Sect.~2]{GM-hydrogenoid-2018}, which we refer to for the details. The special value of the shift \eqref{eq:oneshift} was chosen in \cite{GM-hydrogenoid-2018} in order to be able to solve the ODE $\mathcal{S}_\nu^*u=0$ by means of special functions, this way characterising explicitly the deficiency space $\ker\mathcal{S}_\nu^*$. The Friedrichs extension $\mathcal{S}_{\nu,F}$ of $\mathcal{S}_\nu$ was also characterised in \cite[Sect.~2]{GM-hydrogenoid-2018}. This provides all the ingredients to investigate the intersection \eqref{eq:maincond} and apply  Theorems \ref{thm:main1}-\ref{thm:main2} so as to reproduce \eqref{eq:examplesamebottom}.

%

\def\cprime{$'$}

\end{document}